\documentclass[a4paper,reqno]{amsart}

\usepackage{amsthm, amssymb} \usepackage{amsfonts}
\usepackage{amsmath} \usepackage{mathrsfs} \usepackage[sans]{dsfont}
\usepackage[colorlinks=true,citecolor=black,linkcolor=black,
backref=page]{hyperref}
\usepackage{etoolbox}
\makeatletter
\patchcmd{\BR@backref}{\newblock}{\newblock[cited on p.~}{}{}
\patchcmd{\BR@backref}{\par}{]\par}{}{}
\makeatother

\usepackage[all]{hypcap} \usepackage{mathtools}
\usepackage{thmtools}
\usepackage{cleveref}

\usepackage{kantlipsum} 

\usepackage[normalem]{ulem} 

\usepackage{enumitem}

\theoremstyle{theorem} 
\declaretheorem[name=Theorem]{thm} 
\declaretheorem[name=Proposition]{prop}
\declaretheorem[name=Lemma]{lem}

\theoremstyle{definition} \declaretheorem[name=Definition]{defin}
\theoremstyle{remark} \declaretheorem[name=Remark]{rem}

\newcommand\cC{{\mathcal C}}

\newcommand\cO{{\mathcal O}} 
\newcommand\cP{{\mathcal P}}

 \newcommand\bN{{\mathbb N}}
 \newcommand\bP{{\mathbb P}}
 \newcommand\bR{{\mathbb R}}
\newcommand\bT{{\mathbb T}} \newcommand\bZ{{\mathbb Z}}

\newcommand{\norm}[1]{\left\lVert{#1}\right\rVert}
\newcommand{\abs}[1]{\left\lvert{#1}\right\rvert}

\newcommand{\map}{T}
\newcommand\leb{{\mathbf m}}
\newcommand\metr{\mathbf{d}}
\newcommand\bp{\mathbf{0}}
\newcommand\pnbd{O}
\newcommand\wholetail{Q}
\newcommand\trn{N}
\newcommand\reg{A}
\newcommand\imreg{C}
\newcommand\symimreg{D}

\newcommand{\rtzeroone}{{\varphi_{0,1}}}
\newcommand{\rtonetwo}{{\varphi_{1,2}}}
\newcommand{\rtzerotwo}{{\varphi_{0,2}}}

\newcommand{\mapone}{{T_{1}}}
\newcommand{\maptwo}{{T_{2}}}

\newcommand\diam{\operatorname{diam}}
 
\newcommand\corr{\operatorname{cor}}

\numberwithin{equation}{section}

\author{Peyman Eslami and Carlangelo Liverani}
\title[Almost Anosov]{Mixing rates for symplectic almost Anosov maps}
\date{16 March 2020.}
\address{
  Peyman Eslami, Carlangelo Liverani\\
  Dipartimento di Matematica\\
  II Universit\`{a} di Roma (Tor Vergata)\\
  Via della Ricerca Scientifica, 00133 Roma, Italy.} \email{{\tt
  peslami7@gmail.com}, liverani@mat.uniroma2.it}
  \thanks{ The authors acknowledges the MIUR Excellence Department Project awarded to the Department of Mathematics, University of Rome Tor Vergata, CUP E83C18000100006 which supported P.E.; while C.L. was partially supported by the PRIN Grant {\em Regular and stochastic behaviour in dynamical systems} (PRIN 2017S35EHN)}
\begin{document}

\begin{abstract} 
We establish sharp bounds on the mixing rates of a class of two dimensional non-uniformly hyperbolic symplectic maps. This provides a primer on how to investigate such questions in a concrete example and, at the same time, it solves a controversy between previous rigorous results and numerical experiments.
\end{abstract}

\maketitle

\section{Introduction}
The study of decay of correlations for dynamical systems is a problem of paramount physical relevance. While the uniformly hyperbolic case is by now well understood and have been studied with very precise results, in the non-uniformly hyperbolic case there are still plenty of open problems. The basic idea to treat such cases is to induce. While many inducing schemes exist, the basic reference is the work of Lai Sang Young \cite{You1,You2}. This has allowed to obtain important and fairly general results, at least for rank one attractors, for maps with critical points (e.g see \cite{You4} and reference within). In this paper we consider non-uniformly hyperbolic maps without critical points. The one dimensional expanding map case, starting with \cite{LSV} and continuing with \cite{Sarig1, Sarig2, Gouezel}, 
has witnessed many progressess that have developed into a rather satisfactory theory. Recently such a theory has been extended to important multidimensional expanding examples \cite{HV1, HV2, EMV}. On the contrary the study of non-uniformly hyperbolic maps is still unsatisfactory. In particular, few example have been studied \cite{Hu99, LM, BG06, HuZh} and only recently some general strategies are merging \cite{LT, BMT}.

It seems thus important to work out explicitly some relevant example to see how to develop the theory further. In particular, in \cite{LM} it was introduced a natural class of symplectic maps with a neutral fixed point. When such maps are smooth in \cite{LM} it was proven that smooth observables exhibit a decay of correlations with speed, at least,  $n^{-2}(\log n)^{4}$. However numerical studies \cite{AP} suggest that the decay may be faster, leaving the doubt that the strategy used in \cite{LM} is largely suboptimal. Since the strategy is conceptually the same as used in \cite{LSV}, where it is optimal apart form a logarithmic factor, it is of clear interest to investigate if the suggestion coming from numerics is indeed correct, or it is a numerical artefact. 

The present paper shows that the results in \cite{LM} are indeed sub-optimal and that the correlations decay with a faster power law. Moreover, we show that the power law that we obtain is optimal whereby putting to rest any previous doubt on the correct behaviour of the system. 

The strategy used highlights the key ingredients that are necessary in order to achieve similar sharp results in different systems. 
In particular, it should be mentioned that the map considered is analytic and has a Markov partition and we take advantage of these facts in order to simplify the argument and present it in the simplest possible form. Yet, no real conceptual obstacle prevents one from trying to apply a similar strategy to a non-Markov or a piecewise $\cC^{1+\alpha}$ map, where an array of different power-law decays should be present, see Remark \ref{rem:general}.

The paper is organised as follows: in section \ref{sec:prelim} we present the class of maps we investigate, state our main result (Theorem \ref{thm:main}) and we recall some relevant facts from the literature. In section \ref{sec:MP} we describe various Markov partitions used in the following. Section \ref{sec:firstret}  is devoted to the careful study of the return time when inducing on a set away from the fixed point. This is the core of the paper and, beside treating the current example, highlights the ingredients needed to obtain such results in more general models. At last, in section \ref{sec:mixrates} we use the estimates of section \ref{sec:firstret}, several facts from \cite{LM} and the general theory put forward in \cite{BMT} to prove Theorem \ref{thm:main}.

\begin{rem}[Notation]
Let $f, g: \bT^2 \to \bR$ be two functions. We write $f \ll g$ or $f = \cO(g)$ if there exist constants $C, \delta>0$, depending only on the map  \eqref{eq:map} and the Markov partition defined in Section \ref{sec:MP}, such that for every $x, y \in (0, \delta)$, $\abs{f(x,y)} \le C \abs{g(x,y)}$. We write $f \asymp g$ if
$f \ll g$ and
$g\ll f$.
\end{rem}

\section{Preliminaries and results} \label{sec:prelim}
We consider the class of maps $\map:\bT^{2} \circlearrowleft$ from \cite{LM} and defined by
\begin{equation} \label{eq:map} 
  \map(x,y) = (x+h(x)+y, h(x)+ y),
\end{equation}
where $h \in \cC^{\infty}(\bT^{1}, \bT^{1})$. We moreover require the following properties
\begin{enumerate}
\item $h(0)=0$  (zero is a fixed point);
\item $h'(0)=0$ (zero is a neutral fixed point)
\item $h'(x)>0$ for each $x\neq 0$ (hyperbolicity)
\end{enumerate}
Indeed condition (3) implies that, setting $K_0=\{(v_1,v_2)\in\bR^2\;:\; v_1v_2\geq 0\}$, we have that, for all $(x,y)\in\bT^2$, $D_{(x,y)}\map K_0\subset K_0$. Moreover, $D_{(x,y)}\map K_0\subset \textrm{int} K_0\cup\{0\}$ for all $x\neq 0$. Since $\map(0,y)=(y,y)$, it follows that, if $(x,y)\neq 0$, then $D_{(x,y)}\map^2 K_0\subset \textrm{int} K_0\cup\{0\}$. Hyperbolicity follows then by \cite[Theorem 2.2]{Wo}.

Note that conditions (2--3) imply that zero is a minimum for $h'$, which forces 
\begin{equation*}
h''(0)=0;\quad h'''(0)\geq 0.
\end{equation*}
We will restrict to the generic case
\begin{enumerate}
\item[(4)] $h'''(0)>0$.
\end{enumerate}
Hence, $h$, can be written, in a neighborhood of zero,  as
\begin{equation*} h(x)=bx^{3}+\cO(x^{5}), \text{ for some } b>0.
\end{equation*}

To simplify the following arguments we also assume
\begin{equation*} h(-x)= -h(x).
\end{equation*}
Note that this implies $\map$ is reversible (a physically meaningful property) with respect to the
transformations
\begin{equation} \label{eq:reversible}
  \Pi(x,-y-h(x)):=(x, -y-h(x)); \quad \Pi_1(x,y):=(-x, y+h(x)),
\end{equation} in the sense that $\Pi^2=\Pi_1^2 = \mathbf{Id}$ and
$\Pi \map \Pi = \Pi_1 \map \Pi_1 = \map^{-1}$.

For the record, 
\begin{equation} \label{eq:invmap}
    \map^{-1}(x,y) = (x-y, y-h(x-y)).
\end{equation}

\begin{rem} \label{rem:general} The choice $h\in\cC^\infty$ is rather arbitrary. Most of hyperbolic theory applies to the case $h\in\cC^{1+\alpha}$, $\alpha\in\bR$, $\alpha>0$. For example one could consider the cases in which, near zero, $h'(x)\sim |x|^\alpha$, possibly keeping the symmetry condition $h(-x)=-h(x)$. This would yield a large range of different behaviours of the decay of correlations, probably in analogy with what happens in the one dimensional case. We do not pursue this venue here since it requires a considerable amount of extra work. In particular, one would have to extend all the relevant results obtained in \cite{LM} to the present case. However, what we do in the following constitutes a roadmap toward such an extension.
\end{rem}

 We are interested in studying the correlations between two observables:
\begin{equation*}
  \corr(\Phi, \Psi, n) = \int_{\bT^{2}} \Phi
  \cdot \Psi \circ \map^{n} \,d\leb -
  \int_{\bT^{2}} \Phi \,d\leb\int_{\bT^{2}} \Psi \,d\leb.
\end{equation*}
 Our main result consists in the following sharp estimate.
\begin{thm} \label{thm:main}
For all $\eta >0$ sufficiently close to $1$, there exist
    $C_{1}, C_{2} >0$ such that, for every $\Phi, \Psi \in\cC^{\eta}(\bT^{2})$ satisfying $\int_{\bT^{2}} \Phi
\,d\leb\int_{\bT^{2}} \Psi \,d\leb=1$ (and supported away from $\bp$ for the lower bound), the following estimate holds true. For every $n \ge 1$,
  \begin{equation} \label{eq:sharp}
 C_{1}\frac{\norm{\Phi}_{\cC^\eta} \norm{\Psi}_{\cC^\eta}}{(\log n)n^{3}} \le \abs{\corr(\Phi, \Psi, n)} \le
C_{2 }\frac{(\log n)^{4}\norm{\Phi}_{\cC^\eta} \norm{\Psi}_{\cC^\eta} }{n^{3}}.
\end{equation} 
While, for every $\beta<4$ the exists a constant $C_\beta>0$ such that, for all $\Phi \in\cC^{\eta}(\bT^{2})$ with $\int_{\bT^{2}} \Phi \,d\mu= 0$ and $\Psi \in\cC^{\eta}(\bT^{2})$, we have
\begin{equation*}
\abs{\corr(\Phi, \Psi, n)} \le C_\beta n^{-\beta} \norm{\Phi}_{\cC^\eta}  \norm{\Psi}_{\cC^\eta}.
\end{equation*}
\end{thm}

\begin{rem} The above result shows that, as suggested by some numerical experiments in \cite{AP}, the estimate in \cite{LM} was off by one full power. Moreover, \eqref{eq:sharp} shows that our estimate is essentially sharp.
\end{rem}  
To prove Theorem \ref{thm:main} it is necessary both a better understanding of the hyperbolic structure of the map and precise estimates on the behaviour of the map near its neutral fixed point.

The first is achieved by constructing a drastic refinement of the invariant cone field $K_0$:
There exists two constants $K_{+},K_{-} >0$ such that the unstable
direction $(1,u)$ at the point $(x,y)$ satisfies

\begin{equation} \label{eq:ucone}
K_{-}\left(\abs{x}+\sqrt{\abs{y}}\right)  \le u \le
K_{+}\left(\abs{x}+\sqrt{\abs{y}}\right).
\end{equation}
 Indeed, the unstable direction at $(x,y)=\xi$ must belong to the cone $D_{\map^{-k}\xi}\map^kK_0$, for each $k\in\bN$. But, using \eqref{eq:invmap}, $D_{\map^{-1}\xi}\map^1K_0$ is contained in the cone with boundary lines  $(1+h'(x-y), h'(x-y))$ and $(1,1)$, so provided $|x-y|\geq \delta$, for some fixed $\delta$, the claim is obvious for $K_+<1$ and $K_-$ small enough. On the other hand, if $|x-y|\leq \delta$, then the lower boundary of $D_{\map^{-2}\xi}\map^2K_0$ is above $(1+h'(x-2y+h(x-y)), h'(x-2y+h(x-y)))$. Since $|x-2y+h(x-y)|\geq |y|-(1+\norm{h'}_\infty)\delta$ we have again the claim provided $|y|\geq 2(1+\norm{h'}_\infty)\delta$ and $K_-$ is small enough. It remains to verify the statement in a $\delta$ neighborhood of zero, which is done in \cite[Proposition 4.1]{LM}.

A similar statement holds for the stable direction. 

As for the dynamics near zero, the first task is to understand the shape of the trajectories. This can be achieved with the introductions of an almost conserved quantity: a {\em quasi hamiltonian}.

For an initial point $(x,y)\in \bT^{2}$ and $n \in \bN$, denote
$(x_{n},y_{n})=\map^{n}(x,y)$. By \eqref{eq:map},
\begin{equation} \label{eq:yn} x_{n+1}-x_{n}=y_{n} + h(x_{n}) =
y_{n+1}.
\end{equation}
We define the quasi-Hamiltonian as
  \begin{equation} \label{eq:ham} 
  H(x,y)=\frac12 y^{2} - G(x)+\frac12
h(x)y-\frac1{12}h'(x)y^{2}+\frac1{12}h(x)^{2},
\end{equation} where $G(x) = \int_{0}^{x}h(z)\,dz$.  
 Note that $G\geq 0$, $H(0,0)=0$. A direct computation (if lazy see \cite[Footnote 5]{LM}) yields, for every $(x,y) \in \bT^{2}$,
\begin{equation} \label{eq:quasiham}
\abs{H(\map(x,y))-H(x,y)} \ll x^{8}+y^{4}.
\end{equation}

The dynamics along the trajctories is obviously dominated by dynamics on the stable and unstable manifolds of zero. The bounds on the cone filed (as well as \eqref{eq:ham}, \eqref{eq:quasiham}) imply that they look like parabolas.  The following Lemma corresponds to \cite[Lemmata 3.1, 3.2]{LM}.

\begin{lem}[Dynamics on the stable manifold] \label{lem:stabledyn}
Denote $A =  (2/b)^{1/2}$. Suppose $x_{0}\ge 0$ is sufficiently small. Then,
  there exists a trajectory
  $(x_{n},y_{n})=\map^{n}(x_{0},y_{0})$, $n \in \bN \cup \{0\}$, that
  satisfies $(x_{n})_{n}
  \in \ell^{2}(\bN)$ and
  \begin{equation} \label{eq:stabledyn}
\left|x_{n}-\frac{A}{n+A/x_{0}}\right| \le
\frac{B}{(n+A/x_{0})^{2}}, \text{ for some } B \text{ and all } n \in
\bN;
\end{equation} moreover, $y_{0}$ is a Lipschitz function of
$x_{0}$. These trajectories form the local stable manifold of
$\bp=(0,0)$.
\end{lem}

\section{Markov partitions}
\label{sec:MP}
Using the symmetry (by which we mean reversiblility according to
\eqref{eq:reversible}), we can form a Markov
partition for $\map$ consisting of three elements, as shown in
\Cref{fig:MP}.

\begin{figure}[ht]
\centering
\includegraphics[width = 0.8\columnwidth, height = 0.8\columnwidth,
keepaspectratio]{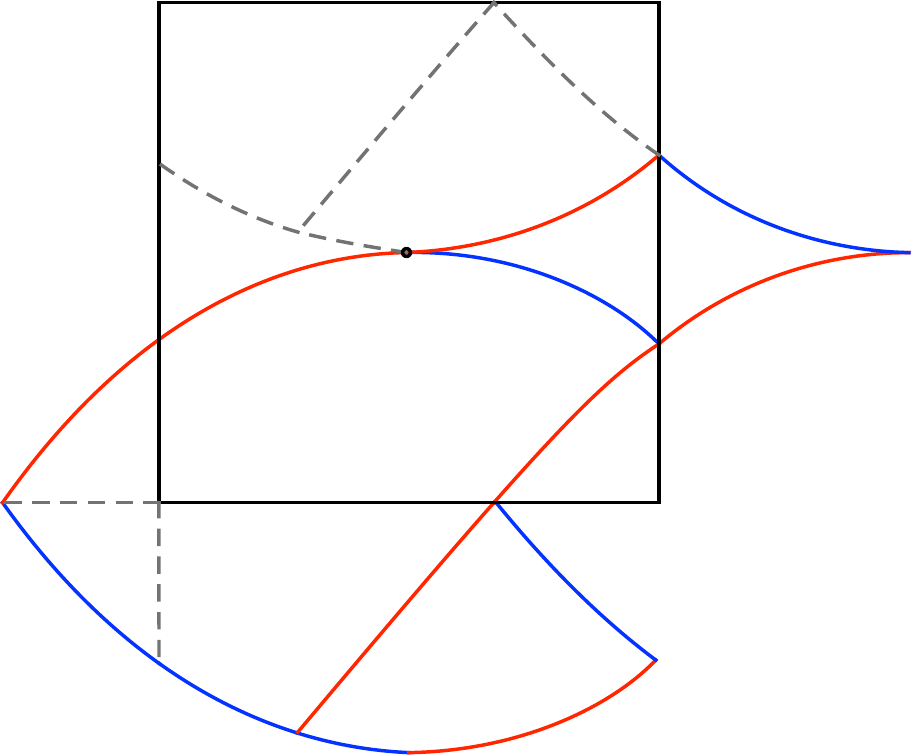}
\caption{Fundamental domain of the Markov partition consisting of
  three rectangles. Note that the regions sticking out of the square
  can be $\bZ^{2}$-translated back into the square to fill up the
  whole square. The blue (decreasing) and red (increasing) curves are pieces of the stable
  and unstable manifolds of $\bp$, respectively.}
\label{fig:MP}
\end{figure} 

 We want to induce on a set away from zero; however, all the above partition elements touch zero. Hence, we refine the original partition forwards and backwards $m$ times. 
We denote such a partition by $\cP$. Let $\pnbd$ denote the union of all the elements of $\cP$ touching $\bp$. By choosing $m$ large enough we ensure that $\pnbd$ is contained in a fixed but sufficiently small neighbourhood of $\bp$.

It is easy to see that $\pnbd$ is necessarily a union of four elements of $\cP$ each of which have two sides consisting of local stable and unstable manifolds of $\bp$. Since $\Pi_1$ maps the local unstable manifold of $\bp$ to its local stable manifold, by continuity, if $P_0$ is one of the three elements of the original Markov partition, then $\Pi_1 P_0 = P_0$. Now, suppose $P \in \cP$ is one of the four sets constituting $\pnbd$. It is of the form 
\begin{equation*}
  P=P_0 \cap T^{-1}P_1 \cap \dots \cap T^{-m}P_m\cap T \tilde P_1 \cap \dots \cap T^m \tilde P_m,
\end{equation*}
where each $P_j$ and $\tilde P_j$ is one of the three elements of the original Markov partition. Using the reversibility property of $\map$ with respect to $\Pi_1$, it follows that $\Pi_1 P$ is of the same form. Since $\Pi_1 P$ touches $\bp$ and is a member of $\cP$, it follows that it is one of the four elements constituting $\pnbd$. It follows that $\Pi_1 \pnbd = \pnbd$.

\section{First return map}
\label{sec:firstret} 
Denote $Y=\bT^{2}\setminus \pnbd$ and define $\rtzeroone:Y \to \bN$ to be the first return time to $Y$. The
first return map $\mapone:Y \circlearrowleft$ is then defined by
$\mapone=\map^{\rtzeroone}$. Note that $Y$ is the union of elements of $\cP \setminus
\{\pnbd\}$ and if we let $R_{\trn} = \{\rtzeroone = \trn\}$, then each $R_{\trn}$ is a union of elements of $\cP^{\trn}$. Let $\cP_1$ denote the partition of $Y$ whose elements are of the form $P \cap R_{\trn}$, where $P \in \cP\setminus \{\pnbd\}$ and $N \in \bN$. Notice that $\cP_1$ is a (countably infinite) Markov partition for $T_1$. 

Denote $\wholetail =
\map^{-1}\pnbd\setminus
\pnbd$. Then, $\wholetail = \bigcup_{\trn \ge 2} R_{\trn}$. Our aim
in this section is to estimate $\leb(R_{\trn})$ for large
$\trn$ (see \Cref{rem:largeN}).

For large $\trn$, $R_{\trn}$ consists of two parts one in the
second quadrant ($x\le 0 , y \ge 0$) and the other in the fourth
quadrant ($x\ge 0, y \le 0$). Due to the symmetry we focus on the
part of
$R_{\trn}$ contained in the second quadrant. We further consider two
cases one corresponding to the part above the stable manifold of
$\bp$ (fat region);  the other corresponding to the part below the
stable manifold of $\bp$ (thin region).

\subsection{Analysis in the fat region}
\label{subsec:fat} First we analyze the dynamics in the fat region. In
\Cref{subsec:thin} we do a similar analysis for the thin region.

Fix $M>0$ large (according to \Cref{lem:largeM}) and let
\begin{equation*}
  \bP_{M}=\{(x,y)\in \bT^{2} : Mx^{4} \le H(x,y), x\le 0, y \ge 0\}.
\end{equation*}

Due to the parabolic nature of stable and unstable manifolds \cite[Lemma~3.5]{LM}, in the fat region it holds $x^2 \ll y$. However, a better estimate holds for $(x,y) \in \bP_M$ by taking $M$ large.

\begin{lem}\label{lem:xy}  For all $(x,y) \in \bP_{M}$,
  \begin{equation} \label{eq:xy} 
    x^{2} \ll (1/M)y.
  \end{equation}
\end{lem}

\begin{proof} From \eqref{eq:ham}, $H(x,y) \le (1/2)y^2+(1/2)h(x)^2 \ll y^2$. Also, by assumption, $Mx^{4} \le H(x,y)$. The result follows.
\end{proof}

\begin{lem} \label{lem:largeM} For $M$ sufficiently large and $(x,y) \in \bP_{M}$,
    \begin{equation} \label{eq:Hest}  
        H(x,y) \asymp y^{2}.
    \end{equation}
  \end{lem}

\begin{proof} By \eqref{eq:ham}, $H(x,y) \ge (1/2)y^2-G(x)+(1/2)h(x)y-(1/12)h'(x)y^2$. By \eqref{eq:xy} this expression is $\gg (1/2)M^2x^4 - x^4$. Choosing $M$ sufficiently large implies the lower bound in \eqref{eq:Hest}. The upperbound was shown in the proof of \Cref{lem:xy}. 
\end{proof}

\begin{lem} For $(x,y) \in \bP_{M}$,
  \begin{equation} \label{eq:sqrtHdiff}
    \abs{H^{1/2} \circ \map (x,y)-H^{1/2}(x,y)} \ll y^{3}.
  \end{equation}
\end{lem}

\begin{proof} By \eqref{eq:Hest} and \eqref{eq:quasiham} for $(x,y)
  \in \bP_{M}$,
  \begin{equation*}
    \begin{split}
      \abs{H^{1/2}\circ \map (x,y)-H^{1/2} (x,y)} &= \frac{\abs{ H \circ
          \map (x,y)-H(x,y)}}{H^{1/2} \circ \map (x,y)+H^{1/2}(x,y)}\\ &\ll
          \frac{\abs{ H \circ \map (x,y)-H(x,y)}}{y} \\ &\ll
          \frac{x^{8}+y^{4}}{y} \ll y^{3}.
    \end{split}
  \end{equation*}
\end{proof}

\begin{defin} \label{def:nell} Given an initial point $(x,y) \in
\wholetail$ let
  $E_{k} = E_{k}(x,y) = H(x_{k}, y_{k})$,
  \begin{equation*}n = \max{\{k \in \bN: x_{k} \le 0\}}, \quad \ell=\ell(x,y) =
    \min\{k \le n : (x_{k}, y_{k}) \in \bP_{M}\}.\end{equation*}
\end{defin}

\begin{rem}
    Note that in the estimates below both sides of $\ll$ are functions of the initial point $(x,y)$ so the constants hidden in the notation $\ll$ do not depend on $k$.
\end{rem}
\begin{lem} For all $\ell \le k \le n$,
  \begin{equation} \label{eq:Esqrt_ch}
    \abs{E_{k}^{1/2}-E_{\ell}^{1/2}} =
    \abs{H^{1/2}(x_{k},y_{k})-H^{1/2}(x_{\ell},y_{\ell})} \ll
    \abs{y_{\ell}}^{2}\abs{x_{\ell}}
  \end{equation}
\end{lem}

\begin{proof} By \eqref{eq:sqrtHdiff},
  \begin{equation} \label{eq:sqrtEdiff}
    \begin{split}
      \abs{E_{k}^{1/2}-E_{\ell}^{1/2}} &=
      \abs{H^{1/2}(x_{k},y_{k})-H^{1/2}(x_{\ell},y_{\ell})} \\ &\ll
      \sum_{j=\ell}^{k-1}\abs{y_{j}}^{3} \ll
      \abs{y_{\ell}}^{2}\sum_{j=\ell}^{k-1}(x_{j}-x_{j-1}) \ll
      \abs{y_{\ell}}^{2}\abs{x_{\ell}}.
    \end{split}
  \end{equation}
\end{proof}

\begin{rem} \label{rem:xkEk} If $(x_{k},y_{k}) \in \bP_{M}$, then
  by definition of $\bP_M$ and \eqref{eq:Hest},
  \begin{equation} \label{eq:ykEk} 
    \abs{x_{k}} \le M^{-1/4}E_{k}^{1/4}, \text{ and } y_{k} \asymp E_{k}^{1/2}.
  \end{equation}
\end{rem}

\begin{lem} \label{lem:precision} For $\ell \le k \le n$,
  \begin{equation} \label{eq:precision} 
    y_k \asymp E_\ell^{1/2}.
  \end{equation}
\end{lem}

\begin{proof} Write
  $y_{k} = y_{\ell} - (y_{\ell}-y_{k})$ and apply
  \eqref{eq:Esqrt_ch} and \eqref{eq:ykEk}. 
\end{proof}

\begin{lem}
  \begin{equation} \label{eq:ellminusone} E_{\ell-1}^{1/4}-
    E_{\ell}^{1/4} \ll E_{\ell}^{1/2}.
  \end{equation}
\end{lem}
\begin{proof} By definition of $\ell$,
  $M^{-1/4}E_{\ell-1}^{1/4} \le -x_{\ell-1} $ and
  $M^{-1/4}E_{\ell}^{1/4} \ge -x_{\ell}$. Therefore, using \eqref{eq:yn} and \eqref{eq:ykEk},
  \begin{equation*} M^{-1/4}E_{\ell-1}^{1/4} - M^{-1/4}E_{\ell}^{1/4} \le
    -x_{\ell-1}+x_{\ell} = y_{\ell} \ll E_{\ell}^{1/2}.
  \end{equation*}
\end{proof}

Now we relate the time and the energy.
\begin{lem} \label{lem:timeout}
  $ n-\ell \asymp E_{\ell}^{-1/4} $.
\end{lem}
\begin{proof} By \Cref{rem:xkEk} and \Cref{lem:precision},
  \begin{equation*}
    \begin{split} (E_{\ell}/M)^{1/4} &\ge -x_{\ell} \ge x_{n}-x_{\ell}
      =
      \sum_{k=\ell+1}^{n}y_{k} \gg (n-\ell-1) E_{\ell}^{1/2} \\
      (E_{\ell-1}/M)^{1/4} &\le
          -x_{\ell-1}\le x_{n+1}-x_{\ell-1} =
      \sum_{k=\ell}^{n+1}y_{k} \ll 
      (n-\ell+1)E_{\ell}^{1/2}.
    \end{split}
  \end{equation*}
  The result follows because, by
  \eqref{eq:ellminusone},
  \begin{equation*} E_{\ell-1}^{1/4} =
    E_{\ell}^{1/4}+(E_{\ell-1}^{1/4}-E_{\ell}^{1/4}) \ge
    E_{\ell}^{1/4} - \cO(E_{\ell}^{1/2}) \gg E_{\ell}^{1/4}.
  \end{equation*}
\end{proof}

\begin{lem} \label{lem:xell} $\abs{x_{\ell}} \gg E_{\ell}^{1/4}$.
\end{lem}
\begin{proof} Using \eqref{eq:yn} and \eqref{eq:precision},
  \begin{equation*}
    \begin{split}
      \abs{x_{\ell}} = \abs{x_{\ell-1}+y_{\ell}} &\ge
      \abs{x_{\ell-1}}-\abs{y_{\ell}} \\ &\ge
      (E_{\ell-1}/M)^{1/4}-\cO(E_{\ell}^{1/2})
      \gg E_{\ell-1}^{1/4}\gg E_{\ell}^{1/4}.
    \end{split}
  \end{equation*}
\end{proof} Next, we relate $\ell$ to $E_{\ell}$. Our strategy is
to show that there exists a point $(x^{s},y^{s})$ on the stable
manifold of
$\bp$ whose trajectory ``shadows'' the trajectory of $(x,y)$ for
$\ell$ iterations; then, we use the dynamics on the stable manifold, \Cref{lem:stabledyn}, 
to estimate $\ell$.

Denote  a sufficiently long piece of the  stable manifold of $\bp$ by $\{(z,\gamma_{s}(z))\}$. Let $\{(z,U_{\ell}(z))\}$ be a sufficiently long piece of the unstable
manifold of the point $(x_{\ell}, y_{\ell})$. 
Denote by $S_{\ell} = (\xi_{\ell}, \zeta_{\ell})$ the point of intersection of
$U_{\ell}$ and $\gamma_{s}$. That is,
$U_{\ell}(\xi_{\ell}) = \gamma_{s}(\xi_{\ell})=\zeta_{\ell}$.

\begin{lem} \label{lem:blah}
  $0 \le \abs{\xi_{\ell}}-\abs{x_{\ell}} \ll E_{\ell}^{1/4}$.
\end{lem}

\begin{proof} The lower bound clearly follows from \eqref{eq:ucone}
which implies that the unstable curves are increasing. To prove
  the upper bound, note that for every $\xi_{\ell}<z < x_{\ell}$,
  \begin{equation*} U_{\ell}'(z) \ge K_{-}(\abs{z}+\sqrt{\abs{U_{\ell}(z)}} )\ge
    K_{-}(\abs{x_{\ell}} +\sqrt{\zeta_{\ell}}) .
  \end{equation*}

  From the geometry and using \Cref{lem:xell} and \Cref{lem:precision} we have
  \begin{equation*}
    \abs{\xi_{\ell}} -\abs{x_{\ell}}\le
    \frac{y_{\ell}-\zeta_{\ell}}{K_{-}(\abs{x_{\ell}}
      +\sqrt{\zeta_{\ell}})} \ll \frac{y_{\ell}}{\abs{x_{\ell}}} \ll
    \frac{E_{\ell}^{1/2}}{E_{\ell}^{1/4}} \ll E_{\ell}^{1/4}.
  \end{equation*}
\end{proof}

\begin{rem} \label{rem:largeN}
  In the following lemma, in addition to previous restrictions on $(x_0, y_0)$ being sufficiently close to $\bp$, we need $\ell$ to be sufficiently large. This is accomplished by considering $\trn$ sufficiently large beacuse this forces $(x_0, y_0) \in R_\trn$ to be sufficently close to the local stable manifold of $\bp$, which in turn, by continuity, forces $\ell$ to be large.
\end{rem}
\begin{lem} \label{lem:timein}
  $ \ell \asymp E_{\ell}^{-1/4}$.
\end{lem}
\begin{proof} Since $\{\xi_{j}\}_{j=0}^{\ell}$ lies on the stable
  manifold of
  $\bp$, by \eqref{eq:stabledyn}, we have
  \begin{equation*}
    \xi_{\ell} \asymp \frac{A}{\ell+A/\xi_{0}} \Rightarrow \ell \asymp
    \frac{A}{\xi_{\ell}} - \frac{A}{\xi_{0}}.
  \end{equation*} Since $\abs{\xi_{\ell}} \ll E_{\ell}^{1/4}$, we have
  $\ell \gg E_{\ell}^{-1/4}$. On the other hand, by
  \eqref{eq:ellminusone}, we have
  \begin{equation*}
    \abs{\xi_{\ell-1}} \ge \abs{x_{\ell-1}} \ge
    M^{-1/4}E_{\ell-1}^{1/4} \gg E_{\ell}^{1/4},
  \end{equation*} which shows that $\ell \ll E_{\ell}^{-1/4}$.
\end{proof}

For $k \in \bN$, define $\reg_k = \{(x,y) \in \wholetail : n(x,y) =
k\}$. This is the set of points in $\wholetail=\map^{-1}\pnbd\setminus
\pnbd$ that spend exactly $k$ iterations on the left of the $y$-axis.
We assume here that $\reg_k$ is in the fat region and in the second
quadrant.

We will need some extra information about the geometry of $\wholetail$: By construction, and properties of the Markov partition, one unstable side of $\wholetail$ is the preimage of the other.
\begin{lem} For sufficiently large $k$, $\reg_k$ is the region bounded
  by the unstable sides of $\wholetail$ and the curves
  $\map^{-k}(\{x=0\})$ and $\map^{-(k+1)}(\{x=0\})$. 
\end{lem}
\begin{proof} For large $k$, the $k$-th preimage of the $y$-axis is a
  curve close to the stable manifold of $\bp$. $\reg_k$ is the region in $\wholetail$ that lies between the curves
  $\map^{-k}(\{x=0\})$ and $\map^{-(k+1)}(\{x=0\})$. 
\end{proof}

Let us denote $\imreg_k = \map^k \reg_k$. For sufficiently large $k$, $\imreg_k$ is the region bounded by two unstable curves, the curve $\map^{-1}(\{x=0\})$ and the $y$-axis.

\begin{figure}[ht]
  \centering
  \includegraphics[width = 0.8\columnwidth]{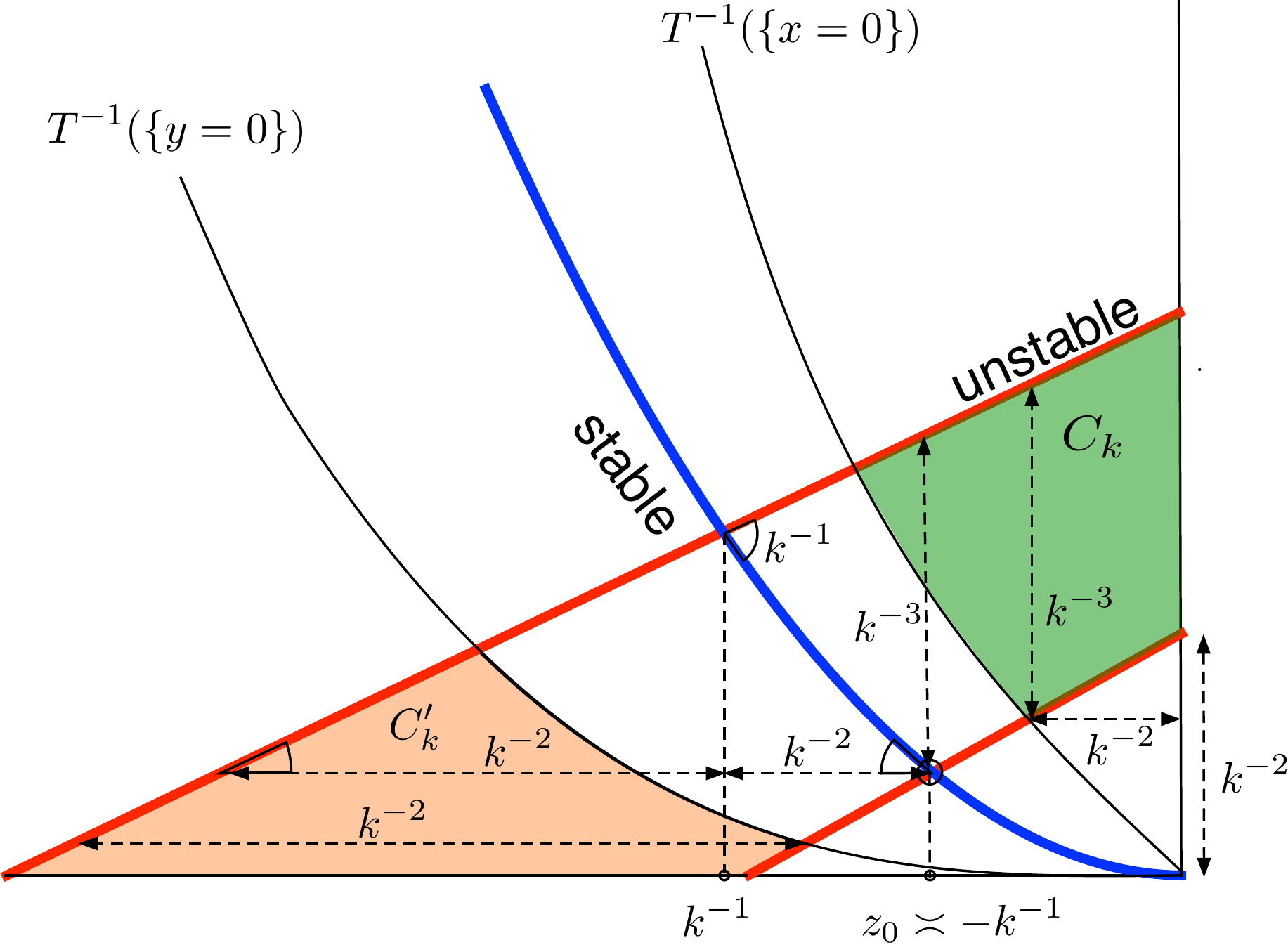}
  \caption{Proof of \Cref{lem:vlength}. Unstable curves are depicted as red increasing straight lines while stable curves (in blue) are decreasing, but in reality they are just smooth increasing and decreasing curves, respectively.}
  \label{fig:LM-gron}
\end{figure}
The following lemma gives an estimate on the vertical distance between the unstable sides of $C_k$; that is, on the length of a vertical line segment both of whose endpoints lie on the unstable sides of $C_k$.
\begin{lem} \label{lem:vlength}
  The vertical distance between the unstable sides of $\imreg_k$ is $\asymp k^{-3}$.
\end{lem}
\begin{proof} The upper and lower boundaries of $\imreg_k$ are formed
  by unstable manifolds that if continued to the left will intersect
  the immediate stable manifold of $\bp$ at points whose $x$-coordinate is proportional to $k^{-1}$ according to \eqref{eq:stabledyn}. Note that one unstable curve is the preimage of the other.  At the points of intersection between the unstable manifolds and the stable manifold (shown in \Cref{fig:LM-gron}), by \eqref{eq:ucone},
  the angle between the manifolds is $\asymp k^{-1}$. Therefore the vertical distance between the two unstable curves
  is $\asymp k^{-3}$ at the intersection of the lower unstable curve and the stable curve (see \Cref{fig:LM-gron}). Let us denote $x$-coordinate of this point by $z_0$. Let us label the two unstable curves by $(x,U_1(x))$ and $(x,U_2(x))$
  and denote their slopes by $u_1(x)$ and $u_2(x)$. By
  the mean value inequality, it follows that
  \begin{equation*}
    \abs{u_2(x)-u_1(x)} \le |Du(\xi)| \norm{(x,U_2(x))-(x,U_1(x))},
  \end{equation*}
  where $\xi \in \bT^2$ is a point on the line segment connecting $(x,U_1(x))$ and $(x,U_2(x))$ and $(1,u(\xi))$ is the unstable vector at $\xi$.
  
  By \cite[Lemma 6.6 and Proposition 4.1]{LM}, $|Du(\xi)| \ll \theta(\xi)^{-1}$, where $\theta(\xi) = u(\xi)+v(\xi)$ and $(1, -v(\xi))$ is the stable vector at $\xi$. By \eqref{eq:ucone}, $|Du(\xi)| \ll k$ in the region bounded by the stable curve, the two unstable curves and the $y$-axis.
  
 Since $U_j'(x) = u_j(x)$, for $j=1,2$, we have shown that
  \begin{equation*}
    \abs{(U_2(x)-U_1(x))'} \ll k \abs{U_2(x)-U_1(x)}.
  \end{equation*} 
  Now,
  by Gronwall inequality,
\begin{equation*} 
  U_2(z)-U_1(z) \le (U_2(z_0)- U_1(z_0))e^{Ck(z-z_0)}.
\end{equation*}
Recall that $-z_0 \asymp k^{-1}$, so for every $z \in (z_0,
  0)$
\begin{equation}
  U_2(z)-U_1(z) \le k^{-3}e^{Ck(z-z_0)} \ll k^{-3}.
\end{equation} Interchanging the role of $z$ and $z_0$, we also get the lower bound,
\begin{equation*}
  U_2(z)-U_1(z) \ge k^{-3}e^{-Ck(z-z_0)} \ge k^{-3}e^{-C} \gg k^{-3}.
\end{equation*}
\end{proof}

\begin{lem} Suppose $(x,y)$ is a point on the boundary of $\imreg_k$ which also lies on $\map^{-1}\{x=0\}$, then $\abs{x} \asymp k^{-2}$.
\end{lem}
\begin{proof} By \Cref{lem:precision}, 
\Cref{lem:timeout} and \Cref{lem:timein}, $y \asymp k^{-2}$.
  Since $\map^{-1}(\{x=0\})= \{(-y, y+h(y))\}$, it follows that $\abs{x} \asymp k^{-2}$.
\end{proof}

\begin{lem} \label{lem:measreg} $\leb(\reg_k) = \leb(\imreg_k) \asymp k^{-5}$.
\end{lem}
\begin{proof} This is a direct consequence of the previous two lemmas and the $\map$-invariance of $\leb$.
\end{proof}

\begin{lem} \label{lem:incl}
$\abs{\pi_y(\imreg_k)} \asymp k^{-3}$, where $\pi_y$
denotes projection onto the $y$-axis.
\end{lem}
\begin{proof} 
  We established in \Cref{lem:vlength} that the vertical distance between the unstable sides of $\imreg_k$ is $\asymp k^{-3}$. It follows that $\abs{\pi_y(\imreg_k)} \gg k^{-3}$. To show the upper bound, it remains to take care of the inclination of the unstable sides. By \eqref{eq:ucone}, the angle of the unstable
  boundaries of $\imreg_k$ with the horizontal is $\asymp k^{-1}$  which
  means they can increase vertically by a factor of $k^{-3}$ in a
  distance of $k^{-2}$. It follows that $\abs{\pi_y(\imreg_k)} \ll k^{-3}$.
\end{proof}

\begin{rem}
  It follows from the previous lemmas that $\imreg_k$ is contained in
  a true rectangle (not a Markov rectangle) of vertical length $\asymp
  k^{-3}$ and of horizontal length $\asymp k^{-2}$.
\end{rem}

\begin{figure}[ht]
    \centering
    \includegraphics[width = 0.6\columnwidth, height = 0.6\columnwidth]{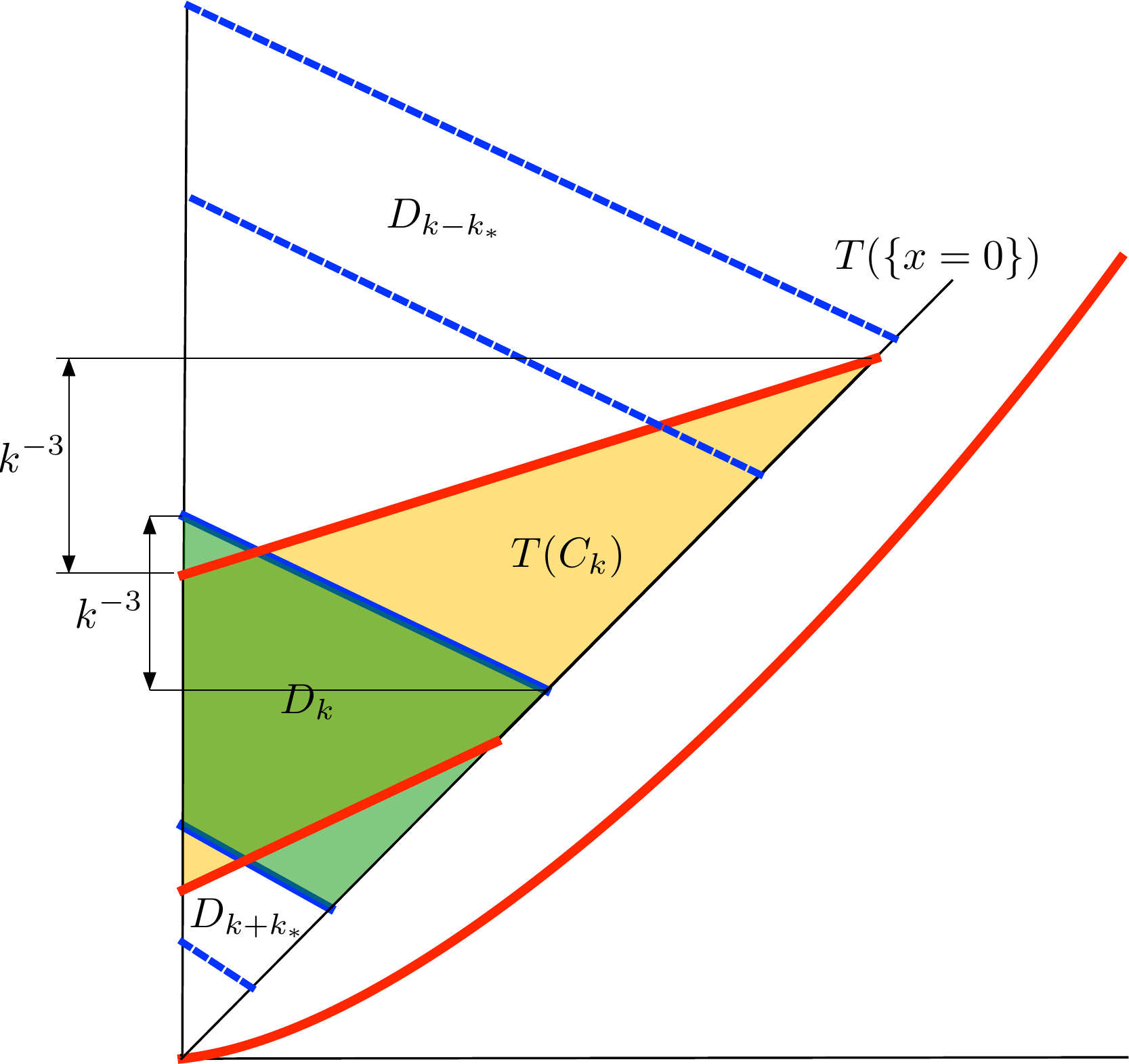}
    \caption{Proof of \Cref{lem:symcover}.}
    \label{fig:LM-gron-sym}
  \end{figure}

Let $\symimreg_k :=\Pi_1 \imreg_k$. Due to the symmetry, $\map^k \symimreg_k = \map^k \Pi_1 \Pi_1 \symimreg_k = \Pi_1 \map^{-k}\Pi_1 \symimreg_k = \Pi_1 \map^{-k} \imreg_k$. Therefore, $\symimreg_k$ is the set of
points on the right of the $y$-axis whose preimage is on the left of
the $y$-axis and that spend exactly $k$ iterations in $\pnbd$ before
mapping into $\map\pnbd\setminus \pnbd$. 

\begin{lem} \label{lem:symcover} There exists $k_* \in \bN$ such that for all sufficiently large $k$,
  \begin{equation*} \map\imreg_k \subset \symimreg_{k-k_*} \cup \dots \cup \symimreg_{k+k_*}.
  \end{equation*}
\end{lem}

\begin{proof} Using the definition of $\map$, $\map\imreg_k$ is vertically lower than
$\imreg_k$ by an amount proportional to $k^{-6}$ and the vertical distance between its unstable sides is $\asymp k^{-3}$. Also $\symimreg_k$
has vertical height proportional to $k^{-3}$ and has inclination (with
respect to the horizontal) of $k^{-3}$ (see proof of \Cref{lem:incl}), as depicted in \Cref{fig:LM-gron-sym}. Since the proportionality constants are independent of $k$ (i.e. they hold for all sufficiently large $k$), $\map\imreg_k$ can be covered by finitely many $\symimreg_k$'s. That is, there exists $k_*$ such that $\map\imreg_k \subset \symimreg_{k-k_*} \cup \dots \cup \symimreg_{k+k_*}$.
\end{proof}

\begin{lem} For sufficiently large $k$,
    \begin{equation*}
      \reg_k \subset R_{2k-k_*} \cup \dots \cup R_{2k+k_*}, \qquad R_{2k}, R_{2k+1} \subset \reg_{k-k_*} \cup \dots \cup \reg_{k+k_*} .
    \end{equation*}
\end{lem}

\begin{proof}
  The first statement follows directly from the previous
  lemma. For the second statement, note that $R_{2k+1}$ spends $2k$
  iterations in $\pnbd$ and this number of iterations is divided between
  left and right sides of the $y$-axis with the restriction that the
  number of iterations on the left and on the right can differ by at
  most $k_*$ because of the previous lemma and the symmetry.  This forces
  at most $j$ iterations on the left and $2k-j$ iterations on the right, where $-k_* \le j \le k_*$. This implies
  that $R_{2k+1}$ is a subset of $\reg_{k-k_*} \cup \dots \cup
  \reg_{k+k_*}$. For the same reason $R_{2k}$ must be a subset of
  $\reg_{k-k_*} \cup \dots \cup\reg_{k+k_*}$.
\end{proof}

\begin{prop} \label{prop:fat}
  \begin{equation*}
    \leb(R_{\trn}^{\text{fat}}) \asymp \trn^{-5}.
  \end{equation*}
\end{prop}

\begin{proof} 
This is a direct consequence of the previous lemma and the estimate on
the measure of $\reg_k$ from \Cref{lem:measreg}.
\end{proof}

\subsection{Analysis in the thin region}
\label{subsec:thin}

Analysis in the thin region is similar to the analysis we did in
\Cref{subsec:fat} for the fat region. Note that in this region the energy is negative.
Fix $M$ large (according to \Cref{lem:Hx}) and define
\begin{equation*}
  \bP'_{M}=\{(x,y)\in \bT^{2} : My^{2} \le \abs{H(x,y)}, x\le 0, y \ge
  0\}.\end{equation*}

  A priori, in the thin region, $y \ll x^2$ but in $\bP'_M$ we have the following better estimate.
 \begin{lem} \label{lem:yxthin} If $(x,y) \in \bP'_{M}$ then
  \begin{equation} \label{eq:yxthin} y \ll (1/M)x^2.
  \end{equation}
\end{lem}
\begin{proof} By \eqref{eq:ham}, $|H(x,y)|=-H(x,y) \le G(x)-(1/2)h(x)y+(1/12)h'(x)y^2 \ll x^4$. Also, by assumption, $My^{2} \le |H(x,y)|$. Together, they imply the result.
\end{proof}

\begin{lem} \label{lem:Hx} For $M$ sufficiently large and $(x,y) \in \bP'_{M}$, 
  \begin{equation} \label{eq:Hx} 
     \abs{H(x,y)} \asymp x^4
  \end{equation}
\end{lem}
\begin{proof}
    By \eqref{eq:ham}, $-H(x,y) \ge -(1/2)y^2+\cO(x^4)$. By \eqref{eq:yxthin}, this expression is $\gg x^4$ if $M$ is chosen sufficiently large. This proves the lower bound in \eqref{eq:Hx}. The upper bound was shown in the proof of \Cref{lem:yxthin}. 
\end{proof}

\begin{lem} For $(x,y) \in \bP'_{M}$,
  \begin{equation} \label{eq:sqrtHdiffx}
    \abs{\abs{H \circ \map (x,y)}^{1/2}-\abs{H(x,y)}^{1/2}} \ll x^{6}.
  \end{equation}
\end{lem}

\begin{proof} By \eqref{eq:Hx} and \eqref{eq:quasiham} for
  $(x,y) \in \bP'_{M}$,
  \begin{equation*}
    \begin{split}
    \abs{\abs{H \circ \map (x,y)}^{1/2}-\abs{H(x,y)}^{1/2}} &=
        \frac{\abs{ H \circ \map (x,y)-H(x,y)}}{\abs{H \circ \map
        (x,y)}^{1/2}+\abs{H(x,y)}^{1/2}}\\ &\ll
        \frac{x^{8}+y^{4}}{x^{2}} \ll x^{6}.
    \end{split}
  \end{equation*}
\end{proof}

As before we define the following quantities but using $\bP'_{M}$
instead of $\bP_{M}$.
\begin{defin} Given an initial point $(x,y) \in \wholetail$ let
  $E_{k} = E_{k}(x,y) = H(x_{k}, y_{k})$,
  \begin{equation*}n = \max{\{k \in \bN: y_{k} \ge 0\}}, \quad \ell=\ell(x,y) =
    \min\{k \le n : (x_{k}, y_{k}) \in \bP'_{M}\}.\end{equation*}
\end{defin}

\begin{lem} For all $\ell \le k \le n$,
  \begin{equation} \label{eq:Esqrt_chx}
    \abs{\abs{E_{k}}^{1/2}-\abs{E_{\ell}}^{1/2}} =
    \abs{\abs{H(x_{k},y_{k})}^{1/2}-\abs{H(x_{\ell},y_{\ell})}^{1/2}}
    \ll \abs{x_{\ell}}^{3}y_{\ell}.
    \end{equation}
\end{lem}

\begin{proof} By \eqref{eq:sqrtHdiffx},
  \begin{equation*}
    \begin{split}
      \abs{\abs{E_{k}}^{1/2}-\abs{E_{\ell}}^{1/2}} &=
      \abs{\abs{H(x_{k},y_{k})}^{1/2}-\abs{H(x_{\ell},y_{\ell})}^{1/2}}
      \\ &\ll
      \sum_{j=\ell}^{k-1}x_{j}^{6} \ll
      \abs{x_{\ell}}^{3}\sum_{j=\ell}^{k-1}(y_{j}-y_{j+1}) \ll
      \abs{x_{\ell}}^{3}y_{\ell}.
    \end{split}
  \end{equation*}
\end{proof}

\begin{rem} \label{rem:xykEk_t} If $(x_{k},y_{k}) \in \bP'_{M}$, then
    by definition of $\bP'_M$ and \eqref{eq:Hx},
    \begin{equation} \label{eq:xykEk_t} 
      \abs{y_{k}} \le M^{-1/2}\abs{E_{k}}^{1/2}, \text{ and } x_{k} \asymp \abs{E_{k}}^{1/4}.
    \end{equation}
  \end{rem}

\begin{lem} \label{lem:precisionx} For $\ell \le k \le n$,
  \begin{equation} \label{eq:precisionx}
    \abs{x_{k}} \asymp \abs{E_\ell}^{1/4}
  \end{equation}
\end{lem}

\begin{proof} Write
  $x_{k} = x_{\ell} + (x_{k}-x_{\ell})$ and apply \eqref{eq:Esqrt_chx} and \eqref{eq:xykEk_t}.
\end{proof}

\begin{lem}
  \begin{equation} \label{eq:ellminusonex} \abs{E_{\ell-1}}^{1/2}-
    \abs{E_{\ell}}^{1/2} \ll \abs{E_{\ell}}^{3/4}.
  \end{equation}
\end{lem}
\begin{proof} By definition of $\ell$,
  $M^{-1/2}\abs{E_{\ell-1}}^{1/2} \le y_{\ell-1} $ and
  $M^{-1/2}\abs{E_{\ell}}^{1/2} \ge y_{\ell}$. Therefore,
  \begin{equation*} M^{-1/2}\abs{E_{\ell-1}}^{1/2} - M^{-1/2}\abs{E_{\ell}}^{1/2} \le
    y_{\ell-1}-y_{\ell} \ll \abs{x_{\ell-1}}^{3} =
    \abs{x_{\ell}-y_{\ell}}^{3} \ll E_{\ell}^{3/4}.
  \end{equation*}
\end{proof}

Now we relate the time and the energy.
\begin{lem} \label{lem:timeinx}
  $ n-\ell \asymp \abs{E_{\ell}}^{-1/2} $.
\end{lem}
  
\begin{proof} By \Cref{rem:xykEk_t} and \Cref{lem:precisionx},
  \begin{equation*}
    \begin{split}
      \abs{E_{\ell}}^{1/2} &\gg y_{\ell} \ge y_{\ell}-y_{n} \gg
      \sum_{k=\ell}^{n-1}\abs{x_{k}}^{3} \gg (n-\ell-1)
      \abs{E_{\ell}}^{3/4} \\
      \abs{E_{\ell-1}}^{1/2} &\ll y_{\ell-1}\le y_{\ell-1}-y_{n+1} \ll
      \sum_{k=\ell-1}^{n}\abs{x_{k}}^{3} =
      (n-\ell+1)\abs{E_{\ell}}^{3/4}.
    \end{split}
  \end{equation*} The result follows because, by
  \eqref{eq:ellminusonex},
  \begin{equation*} \abs{E_{\ell-1}}^{1/2} = \abs{E_{\ell}}^{1/2}+(\abs{E_{\ell-1}}^{1/2}-
    \abs{E_{\ell}}^{1/2}) \gg
    \abs{E_{\ell}}^{1/2}.
  \end{equation*}
\end{proof}

We shall not use the following lemma, nevertheless we state it since it is a symmetrical statement to \Cref{lem:xell}.
\begin{lem}
$\abs{y_{\ell}} \gg \abs{E_{\ell}}^{1/2}$.
\end{lem}
\begin{proof} Using \Cref{rem:xykEk_t} and \eqref{eq:precisionx},
\begin{equation*}
\begin{split}
\abs{y_{\ell}} = \abs{y_{\ell-1}+h(x_{\ell})} &\gg
\abs{y_{\ell-1}}-\abs{x_{\ell}}^{3} \\ &\gg
\abs{E_{\ell-1}}^{1/2}-\cO(\abs{E_{\ell}}^{3/4})
\gg \abs{E_{\ell}}^{1/2}.
\end{split}
\end{equation*}
\end{proof}

Let $\xi_{\ell}, \zeta_{\ell}$ be as in \Cref{lem:blah}.
\begin{lem}
  $0 \le \abs{x_{\ell}}-\abs{\xi_{\ell}} \ll \abs{E_{\ell}}^{1/4}$.
\end{lem}

\begin{proof} The lower bound clearly follows from \eqref{eq:ucone}
  which implies that the unstable curves are increasing. To prove
  the upper bound, note that for every $\xi_{\ell}<z < x_{\ell}$,
  \begin{equation*} U_{\ell}'(z) \ge K_{-}(\abs{z}+\sqrt{\abs{U_{\ell}(z)}} )\ge
    K_{-}(\abs{x_{\ell}} +\sqrt{\zeta_{\ell}}) .
  \end{equation*}

Using \Cref{rem:xykEk_t} and \eqref{eq:precisionx} we have
  \begin{equation*}
    \abs{x_{\ell}}-\abs{\xi_{\ell}} \le
    \frac{y_{\ell}-\zeta_{\ell}}{K_{-}(\abs{x_{\ell}}
    +\sqrt{\zeta_{\ell}})} \ll
    \frac{y_{\ell}}{\abs{x_{\ell}}} \ll
    \frac{\abs{E_{\ell}}^{1/2}}{\abs{E_{\ell}}^{1/4}} \ll \abs{E_{\ell}}^{1/4}.
  \end{equation*}
\end{proof}

\begin{lem} \label{lem:timeoutx}
  $ \ell \asymp \abs{E_{\ell}}^{-1/4}$.
\end{lem}
\begin{proof} Since $\{\xi_{j}\}_{j=0}^{\ell}$ lies on the stable
  manifold of
  $\bp$, by \eqref{eq:stabledyn}, we have
  \begin{equation*}
    \xi_{\ell} \asymp \frac{A}{\ell+A/\xi_{0}} \Rightarrow \ell \asymp
    \frac{A}{\xi_{\ell}} - \frac{A}{\xi_{0}}.
  \end{equation*} Since $\abs{\xi_{\ell}} \le \abs{x_{\ell}}\ll \abs{E_{\ell}}^{1/4}$,
  we get
  $\ell \gg \abs{E_{\ell}}^{-1/4}$. On the other hand, by
  \eqref{eq:ellminusonex}, we have
  \begin{equation*}
    \abs{\xi_{\ell-1}} \ge \abs{x_{\ell-1}} = \abs{x_{\ell} -
    y_{\ell}} \gg
    \abs{E_{\ell}}^{1/4} - \abs{E_{\ell}}^{1/2} \gg
    \abs{E_{\ell}}^{1/4},
  \end{equation*} which shows that $\ell \ll E_{\ell}^{-1/4}$.
\end{proof}

The above lemmas show that we can prove the analog of \Cref{prop:fat}
in the thin region. The proof is very similar, nevertheless we take the reader through it.

For $k \in \bN$, define $\reg_k' = \{(x,y) \in \wholetail : n(x,y) =
k\}$. This is the set of points in $\wholetail=\map^{-1}\pnbd\setminus
\pnbd$ that spend exactly $k$ iterations on the top of the $x$-axis.
We assume here that $\reg_k$ is in the thin region and in the second
quadrant.

\begin{lem} For sufficiently large $k$, $\reg_k'$ is the region bounded
  by the unstable sides of $\wholetail$ and the curves
  $\map^{-k}(\{y=0\})$ and $\map^{-(k+1)}(\{y=0\})$. 
\end{lem}
\begin{proof} For large $k$, the $k$-th preimage of the $x$-axis is a
  curve close to the stable manifold of $\bp$. $\reg_k'$ is the region in $\wholetail$ that lies between the curves
  $\map^{-k}(\{y=0\})$ and $\map^{-(k+1)}(\{y=0\})$. 
\end{proof}

Let us denote $\imreg_k' = \map^k \reg_k'$. For sufficiently large $k$, $\imreg_k'$ is the region bounded by two unstable curves, the curve $\map^{-1}(\{y=0\})$ and the $x$-axis. See \Cref{fig:LM-gron}.

The following lemma gives an estimate on the horizontal distance between the unstable sides of $\imreg_k'$; that is, on the length of a horizontal line segment both of whose endpoints lie on the unstable sides of $\imreg_k'$.
\begin{lem} \label{lem:hlength}
  The horizontal distance between the unstable sides of $\imreg_k'$ is $\asymp k^{-2}$.
\end{lem}
\begin{proof} The upper and lower boundaries of $\imreg_k'$ are formed
  by unstable manifolds that if continued to the right will intersect
  the immediate stable manifold of $\bp$ at points whose $x$-coordinate is proportional to $k^{-1}$ according to \eqref{eq:stabledyn} (and using \eqref{eq:ucone}). Note that one unstable curve is the preimage of the other. Consider the point of intersection of the bottom unstable manifold and the stable manifold shown in \Cref{fig:LM-gron}. The horizontal line segment between this point and the the top unstable manifold consists of two subsegments both of which have length $\asymp k^{-2}$. In fact, the right subsegment has length $\asymp k^{-2}$ due to \eqref{eq:stabledyn}. The left subsegment has also length $\asymp k^{-2}$ because the angles shown in \Cref{fig:LM-gron} at the endpoints of the horizontal segment are both $\asymp k^{-1}$. So the horizontal distance between the two unstable curves at this level 
  is $\asymp k^{-2}$. 
  
  Now a similar argument to the proof of \Cref{lem:vlength} (interchanging the role of $x$ and $y$ axes) implies that the horizontal distance between the unstable curves remains $\asymp k^{-2}$ all the way down to the $x$-axis.
\end{proof}

\begin{lem} Suppose $(x,y)$ is a point on the boundary of $\imreg_k'$ which also lies on $\map^{-1}\{y=0\}$, then $y \asymp k^{-3}$.
\end{lem}
\begin{proof} By \Cref{lem:precisionx}, 
\Cref{lem:timeoutx} and \Cref{lem:timeinx}, $-x \asymp k^{-1}$.
  Since $\map^{-1}(\{y=0\})= \{(x, -h(x))\}$, it follows that $y \asymp k^{-3}$.
\end{proof}

\begin{lem} \label{lem:measregsym} $\leb(\reg_k') = \leb(\imreg_k') \asymp k^{-5}$.
\end{lem}
\begin{proof} This is a direct consequence of the previous two lemmas and the $\map$-invariance of $\leb$.
\end{proof}

\begin{lem} \label{lem:inclsym}
$\abs{\pi_x(\imreg_k')} \asymp k^{-2}$, where $\pi_x$
denotes projection onto the $x$-axis.
\end{lem}
\begin{proof} 
  We established in \Cref{lem:hlength} that the horizontal distance between the unstable sides of $\imreg_k'$ is $\asymp k^{-2}$. It follows that $\abs{\pi_x(\imreg_k')} \gg k^{-2}$. To show the upper bound, it remains to take care of the inclination of the unstable sides. By \eqref{eq:ucone}, the angle of the unstable
  boundaries of $\imreg_k$ with the horizontal is $\asymp k^{-1}$  which
  means they can increase horizontally by a factor of $k^{-2}$ in a vertical
  distance of $k^{-3}$. It follows that $\abs{\pi_x(\imreg_k)} \ll k^{-2}$.
\end{proof}

\begin{rem}
  It follows from the previous lemmas that $\imreg_k'$ is contained in
  a true rectangle (not a Markov rectangle) of vertical length $\asymp
  k^{-3}$ and of horizontal length $\asymp k^{-2}$.
\end{rem}

\begin{figure}[ht]
    \centering
    \includegraphics[width = 0.8\columnwidth, height = 0.8\columnwidth, keepaspectratio]{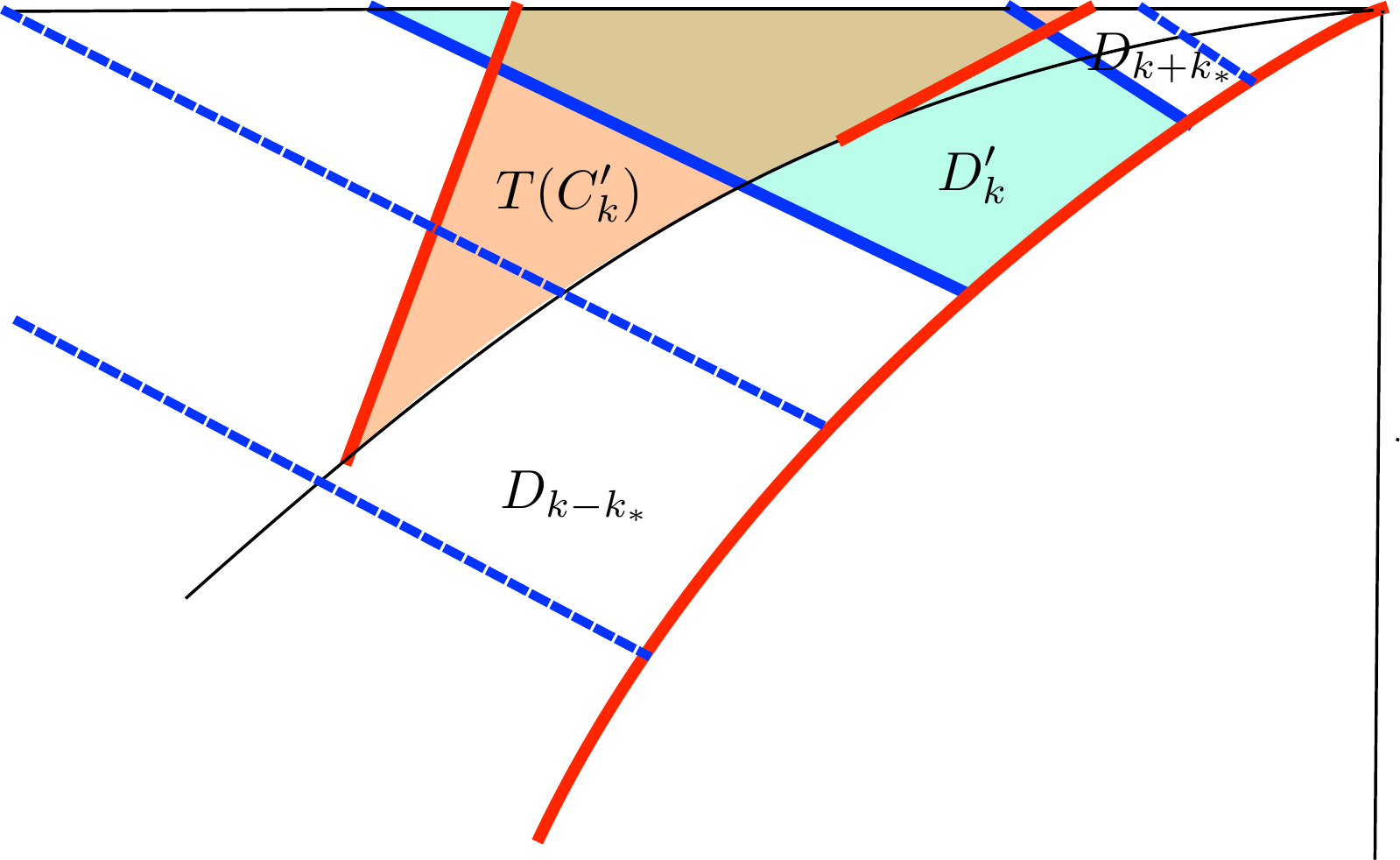}
    \caption{Proof of \Cref{lem:symcoverx}.}
    \label{fig:LM-gron-sym-x}
  \end{figure}

Let $\symimreg_k' :=\Pi_1 \imreg_k'$. Due to the symmetry, $\symimreg_k'$ is the set of
points on the bottom of the $x$-axis whose preimage is on the top of
the $x$-axis and that spend exactly $k$ iterations in $\pnbd$ before
mapping into $\map\pnbd\setminus \pnbd$. 

\begin{lem} \label{lem:symcoverx} There exists $k_* \in \bN$ such that for all sufficiently large $k$,
  \begin{equation*} \map\imreg_k' \subset \symimreg_{k-k_*}' \cup \dots \cup \symimreg_{k+k_*}'.
  \end{equation*}
\end{lem}

\begin{proof} Using the definition of $\map$, $\map\imreg_k'$ is shifted to the left horizontally
by an amount proportional to $k^{-6}$ with respect to $\imreg_k'$ and the horizontal distance between its unstable sides is $\asymp k^{-2}$. Also $\symimreg_k'$
has horizontal length proportional to $k^{-2}$ and has inclination (with
respect to the vertical) of $k^{-2}$ (see proof of \Cref{lem:inclsym}) as depicted in \Cref{fig:LM-gron-sym-x}. Since the proportionality constants are independent of $k$, $\map\imreg_k'$ can be covered by finitely many $\symimreg_k'$s. That is, there exists $k_*$ such that $\map\imreg_k' \subset \symimreg_{k-k_*}' \cup \dots \cup \symimreg_{k+k_*}'$.
\end{proof}

\begin{lem} For sufficiently large $k$,
    \begin{equation*}
      \reg_k' \subset R_{2k-k_*} \cup \dots \cup R_{2k+k_*}, \qquad R_{2k}, R_{2k+1} \subset \reg_{k-k_*}' \cup \dots \cup \reg_{k+k_*}' .
    \end{equation*}
\end{lem}

\begin{proof}
  The first statement follows directly from the previous
  lemma. For the second statement, note that $R_{2k+1}$ spends $2k$
  iterations in $\pnbd$ and this number of iterations is divided between
  top and bottom sides of the $x$-axis with the restriction that the
  number of iterations on the top and on the bottom can differ by at
  most $k_*$ because of the previous lemma and the symmetry.  This forces
  at most $j$ iterations on the top and $2k-j$ iterations on the bottom, where $-k_* \le j \le k_*$. This implies
  that $R_{2k+1}$ is a subset of $\reg_{k-k_*}' \cup \dots \cup
  \reg_{k+k_*}'$. For the same reason $R_{2k}$ must be a subset of
  $\reg_{k-k_*}' \cup \dots \cup\reg_{k+k_*}'$.
\end{proof}

\begin{prop} \label{prop:thin}
  \begin{equation*}
    \leb(R_{\trn}^{\text{thin}}) \asymp \trn^{-5}.
  \end{equation*}
\end{prop}

\begin{proof} 
This is a direct consequence of the previous lemma and the estimate on
the measure of $\reg_k'$ from \Cref{lem:measregsym}.
\end{proof}

Adding the two estimates of \Cref{prop:fat}, \Cref{prop:thin} and symmetrical estimates for other regions, we
get:
\begin{prop} \label{prop:tailmeas}
  \begin{equation*}
    \leb(\{\rtzeroone=N\} \cap \wholetail) \asymp N^{-5}.
  \end{equation*}
\end{prop}

\section{Mixing rates}
\label{sec:mixrates} 
To obtain upper and lower bounds on mixing rates we further induce
\(\mapone\) (the first return map of $\map$ to $Y$) to a two-sided Young tower. This will allows us to apply \cite[Theorem
7.4]{BMT}. 

In order to induce $\mapone$ to a two-sided Young tower with exponential tails, we check conditions (P1)--(P5) of \cite[Section~1]{You1} for $\mapone$. We will see that hese conditions follow from the existence of a finite Markov partition for $\map$ and the (non-uniform) hyperbolicity estimates for $\map$ established in \cite{LM}.

Recall from \Cref{sec:firstret} that $\mapone$ has a countable Markov partition $\cP_1$, which is a certain refinement of the finite parition $\cP \setminus \{\pnbd\}$. We take $\wholetail = \bigcup_{\trn \ge 2} R_{\trn}$ as the set with hyperbolic product structure required by \cite[(P1)]{You1}. We claim that each of the sets $R_{\trn}$ return under iterations of $\mapone$ and u-cross the set $\wholetail$. Furthermore, there is a fixed time $N_*$ before which this happens. This claim follows from the construction of $R_{\trn}$ and the existence of the finite Markov partition for $\map$. Indeed the image of each $R_{\trn}$ under $\mapone$ must u-cross elements of $\cP$. By finiteness of $\cP$, the Markov property and the ergodicity of $\mapone$, in at most finitely many more iterations of $\mapone$, $R_{\trn}$ must also u-cross $\wholetail$ in which case we stop and define the return time accordingly. The leftover is again a union of sets that u-cross elements of $\cP$, so we can repeat the same argumet for them. Since the time between the stopping times is uniformly bounded an due to boundedness of distortion, the tail of the stopping times will be exponentially small in the number of iterations. We have established \cite[(P2)]{You1}.

It remains to establish \cite[(P3)-(P5)]{You1}. 

\textit{Hyperbolicity:} \cite[(P3) and (P4)(a)]{You1} require exponential contraction of $\mapone$ along stable manifolds and backward contraction along unstable manifolds. Both statments follow from the estimates on expansion and contraction rates of the original map $\map$ in the neighbourhood $\pnbd$ of $\bp$. These estimates are obtained in \cite[Lemma~5.1, 5.2 and Corollary~5.3]{LM}. Essentially, a vector in the unstable cone at $\xi \in R_N$ expands by a factor proportional to $N^2$ under $\mapone|_{R_N}=\map^N$. This implies uniform expansion on $\cup_{N=N_0}^\infty R_N$ for some $N_0$. We have also uniform expansion on the other finitely many $R_N$'s by the uniform hyperbolicity of the original $\map$ away from $\bp$. Similar estimate holds for backward expansion along the stable direction \cite[Corollary~5.3]{LM}. Since $\map$ preserves Lebesgue measure, similar estimates hold for contraction along stable and backward expansion along unstable manifolds (see \cite[Lemma~5.2]{LM}).

\textit{Distortion bounds:} Let $J^u\map(\xi) = \norm{D_\xi T (1,u)}/\norm{(1,u)}$ denote the factor of expansion on the unstable manifold of $\xi$. By direction calculation,

\begin{equation*}
    J^u\map(\xi) =  \left( \frac{(1+h'(x)+u(\xi))^2 + (h'(x)+u(\xi))^2}{1+u^2(\xi)} \right)^{1/2}.
\end{equation*}

Since $h$ is smooth and $u$ is differentiable on $\bT^2\setminus \{\bp\}$, with a uniform bound on the derivative, it follows that there exists $C>0$ such that $|\nabla \log J^u\map (\xi)| \le C$. Now, by the mean value inequality, it follows that if $\eta$ is another point on the unstable manifold of $\xi$, then  
\begin{equation*}
    \frac{J^u\map(\xi)}{J^u\map(\eta)} \le e^{C\metr(\xi, \eta)}.
\end{equation*}
The distortion property for $\mapone$ now follows if we can ensure the same inequality along the orbit of $\xi$, $\eta$ as they move through the region $\pnbd$, but this follows by chain rule and the estimate on the expansion factor along the orbit, which is $\ge CN^2$ (for some constant $C>0$) where the points spend $N$ iterations in $\pnbd$. Therefore the distortion property also holds for $\mapone$. Conditions (P4)(b) and (P5)(a) follow from this distortion property and expansion and contraction along stable and unstable directions. Note that the estimates on the slope of stable and unstable manifolds \eqref{eq:ucone} allow one to relate the distance $\metr(\cdot, \cdot)$ to the distance along stable or unstable directions.

\textit{Regularity of the stable holonomy:} Property (P5)(b) of \cite{You1} requires the existence and absolute continuity of the stable holonomy as well as an asymptotic formula for the Jacobian of the holonomy. All of these properties follow from the uniform hyperbolicity and distrotion estimates for $\mapone$ established above. We refer the reader to \cite[Proof of Theorem~3.1]{Mane}.

We have finished checking conditions (P1)-(P5) of \cite{You1} from which follows the existence of a two sided Young tower with exponential tails (for $\mapone$). Let $\rtonetwo$ be the stopping time defined above and $\maptwo = \mapone^\rtonetwo$. Let us denote the full return time by $\rtzerotwo: \wholetail \to \bN$, $\rtzerotwo(x)=\sum_{k=0}^{\rtonetwo(x)-1}\rtzeroone(T_1^k(x))$. That is $\maptwo = \map^\rtzerotwo$. From standard computations $\rtzerotwo$ will have the same tail of return times as $\rtzeroone$ with an additional 
 $[\log(n)]^\alpha$ (for some $\alpha >1$) factor. Note that since $\map$ is mixing, it follows that $\gcd(\rtzerotwo) = 1$.

The article \cite{BMT} requires one more condition. We need to check that there exist $C>0$ and $\theta \in (0,1)$ such that for every $z, z' \in \wholetail$ and $n \ge 1$, 
\begin{equation} \label{eq:sContractionRate}
  \metr(\maptwo^n z, \maptwo^n z') \le C (\theta^n + \theta^{s(z,z')-n}).
\end{equation}

This follows from uniform hyperbolicity. We can choose $C = \diam(\wholetail)$ and $\theta = \max(\lambda_u^{-1}, \lambda_s)$, where $\lambda_u$ and $\lambda_s$ are respectively the worst expansion and contraction of the map $\maptwo$. Then \eqref{eq:sContractionRate} follows from $\maptwo$ being a uniformly hyperbolic map.

Now we are in a position to apply \cite[Theorem 7.4]{BMT}. 

\begin{proof}[Proof of \Cref{thm:main}]
  Denote $\bar\varphi_{0,2}=\int_Q \rtzerotwo \,d\leb$. By \cite[Theorem~7.4, Proposition~7.3]{BMT}, for all $\eta >0$ sufficiently close to $1$ there exists a constant $C>0$ such that for every $\Phi, \Psi \in\cC^\eta(\bT^2)$ satisfying $\int_{\bT^{2}} \Phi
  \,d\leb\int_{\bT^{2}} \Psi \,d\leb=1$ (and supported away from $\bp$ for the lower bound)
  \begin{equation}\label{eq:BMT}
    \abs{\corr(\Phi, \Psi, n) -  \bar\varphi_{0,2}\sum_{j>n} \leb(\rtzerotwo>j)} \le C\norm{\Phi}_{\cC^\eta} \norm{\Psi}_{\cC^\eta} (\gamma_n+\zeta_{\beta'}(n)), 
  \end{equation} 
 where 
 \begin{itemize}
   \item $\beta'$ can be taken any number $<\beta$ and $\beta$ is the polynomical rate of decay of $\leb(\rtzeroone>n)$, which for us, by \Cref{prop:tailmeas}, $\beta =4$.
   \item $\gamma_n \le C n^{-\beta'}\log n$,  by \cite[Proposition~3.2]{BMT}.
   \item $\zeta_{\beta'}(n)=n^{-\beta'}$, for any $\beta'>2$ by \cite[equation~(1.3)]{BMT}. 
 \end{itemize}
 Also, by \cite[Proposition~5.1]{BMT} and \Cref{prop:tailmeas},
 \begin{equation*}
   C_1(\log j)^{-1}j^{-4} \le \leb(\rtzerotwo>j) \le C_2 (\log j)^4 j^{-4}.
 \end{equation*}
 Therefore, absorbing $\bar\varphi_{0,2}$ into the constants $C_1, C_2$ (whose values from one occurence to the other are not necessarily the same), \eqref{eq:BMT} implies
 \begin{equation*}
   C_1 \norm{\Phi}_{\cC^\eta} \norm{\Psi}_{\cC^\eta} n^{-3}(\log n)^{-1}\le \abs{\corr(\Phi, \Psi, n)} \le C_2 \norm{\Phi}_{\cC^\eta} \norm{\Psi}_{\cC^\eta} n^{-3}(\log n)^4. 
 \end{equation*} 
 In the case that one of the observables has zero mean, \cite[Theorem~7.4]{BMT} states that $\abs{\corr(\Phi, \Psi, n)}\le C \gamma_n \norm{\Phi}_{\cC^\eta} \norm{\Psi}_{\cC^\eta}$. So $\abs{\corr(\Phi, \Psi, n)} \le C n^{-\beta'}\log n \norm{\Phi}_{\cC^\eta} \norm{\Psi}_{\cC^\eta}$. \Cref{thm:main} is proved.
\end{proof}

\bibliographystyle{abbrv}

\end{document}